\documentclass[]{article}
\usepackage{graphicx}
\usepackage{amssymb}
\usepackage{eufrak,mathrsfs,amsmath,longtable}

\newtheorem{thm}{Theorem}

\newtheorem{example}{Example}
\newenvironment{proof}{\noindent{\bf Proof:}}{$\hfill \Box$ \vspace{10pt}}

\begin{document}

\title{Lie symmetries reduction and spectral methods on the fractional two-dimensional heat equation}
\author{Rohollah Bakhshandeh-Chamazkoti\thanks{Corresponding author}~$^1$, and Mohsen Alipour\thanks{Corresponding author}~$^2$\\[3mm]
{\small Department of Mathematics, Faculty of Basic Sciences, Babol  Noshirvani University of Technology, Babol, Iran.}\\[3mm]
{\small Emails: $^1$r\_bakhshandeh@nit.ac.ir,~ $^2$m.alipour@nit.ac.ir.}
}
\maketitle
\begin{abstract}
In this paper, the Lie symmetry analysis is proposed for a space-time convection-diffusion
fractional differential equations with the Riemann-Liouville derivative by (2+1) independent variables and one dependent variable.
We find a reduction form of our governed fractional differential equation using the similarity solution of our Lie symmetry.
One-dimensional optimal system of Lie symmetry algebras is found. We present a computational method via the spectral method based on
Bernstein's operational matrices to solve the two-dimensional fractional heat equation with some initial conditions.
\end{abstract}
\vspace{0.7cm}
\section{Introduction}
The diffusion equation is one of the well-known equations with many applications in engineering problems, heat conduction, chemical diffusion,
fluid flow, mass transfer, refrigeration, and traffic analysis, and so on, \cite{Davydovych, Oberlack1, Rehman}. Recently, the study of fractional ordinary differential equation
and the partial differential equation has attracted much attention due to an exact description of
differential equations in fluid mechanics, biology, physics, engineering, and other areas of science \cite{Hashemi, Srivastava, Tarasov}.
Lie symmetry analysis method that originally advocated by Sophus Lie plays a
powerful tools to obtain some exact solutions of differential equations, \cite{Bluman, Olver, Ovsiannikov}.
The fundamental idea of the Lie symmetry analysis is regarding the tangent structural equations under one or several parameters Lie groups of point transformations.
One can construct exact solutions including similarity solutions or more general group-invariant solutions by corresponding symmetry reductions, \cite{Nadj1, Nadj2}.
There have been some new generalizations of the classical
Lie group analysis for symmetry reductions. For instance, Ovsiannikov, \cite{Ovsiannikov}, has extended the method of partially invariant solutions. His
works are based on the concept of an equivalence group, which is a Lie transformation group acting in the extended space
that is called jet space, and preserving the class of given partial differential equations. He has introduced $s$-dimension optimal system
for classification in Lie symmetry algebras of a given differential equation.
The efficiency of the Lie symmetry technique for integer-order differential equations, \cite{Bokhari, Gungor, Kumar, Torrisi},
encouraged the scientists for its extension to the fractional differential equation (FDEs) and then
many authors have made their contribution to the applications of symmetry method for the analysis of some FDEs,
\cite{Buckwar, Demiray, Gazizov1, Gazizov2, Kasatkin}.
In most of the available papers, the symmetries of the FDEs with time-fractional in (1+1) independent variables have been analyzed,
\cite{Hashemi2,  Hu, Lukashchuk, Ouhadan, Sahadevan, SahaRay1, SahaRay2, Wang}.
Recently, the symmetry approach has been developed by Singla and Gupta for the complete
group classification of space-time fractional partial differential equation with (2+1) independent variables and two dependent variables, \cite{Singla1, Singla2, Singla3, Singla4}.
In this article, we consider the following (2+1) fractional convection-diffusion equation
\begin{eqnarray}\label{heateq}
D^\alpha_t u(t, x, y)=D^\beta_x u(t, x, y)+D^\beta_y u(t, x, y)+f(x, y,t),
\end{eqnarray}
where $0<\alpha\leq1,~~0<\beta<1$ and $D^\alpha_t $, $D^\beta_x$ and $D^\beta_x$ are the Riemann-Liouville fractional derivatives of order $\alpha$ and $\beta$ with respect to the variable $t, x$ and $y$ and $f$ is an arbitrary smooth function. The Riemann-Liouville fractional derivatives of order $\alpha$
with respect to the variable $x_i$, $i=1, \cdots, n$, are, \cite{Hilfer},
\begin{align}\label{eq:2}
D^\alpha_{x_i} u(x_1, \cdots,x_n)=\left\{
\begin{array}{lll}
\frac{\partial^n {u}}{\partial{x_i}^n} &,& \alpha=n, \\
\frac{1}{\Gamma(n-\alpha)} \frac{\partial^n}{\partial{x_i}^n} \displaystyle{\int^{x_i}_{0}}(x_i-s)^{n-\alpha-1}u(x_1, \cdots, x_{i-1}, s, x_{i+1}, \cdots, x_n)ds &,& 0\leq n-1<\alpha<n,
\end{array}\right.
\end{align}
and the Caputo fractional derivatives are
\begin{align}\label{eq:2}
{}^C_0{D^{\alpha }_{x_i}} u(x_1, \cdots,x_n)=\left\{
\begin{array}{lll}
\frac{\partial^n {u}}{\partial{x_i}^n} &,& \alpha=n, \\
\frac{1}{\Gamma(n-\alpha)} \displaystyle{\int^{x_i}_{0}}(x_i-s)^{n-\alpha-1} \frac{\partial^nu}{\partial{x_i}^n}(x_1, \cdots, x_{i-1}, s, x_{i+1}, \cdots, x_n)ds &,& 0\leq n-1<\alpha<n,
\end{array}\right.
\end{align}
here $n\in {\Bbb N}$, $\frac{\partial^n}{\partial{x_i}^n}$ is the usual partial derivative of integer order $n$ with respect to $x_i$, $1\leq i\leq n$.
One of the most important numerical methods for solving
linear, nonlinear, and fractional equations are Bernstein polynomials which have been used in many papers, \cite{Alipour1, Alipour2, Alipour3, Alipour4, Alipour5, Mirzaee1, Mirzaee2}.
We try to apply the spectral method based on Bernstein's operational matrices on our given fractional differential equation.
Then we compare the exact solution and an approximate solution in an illustrative example for $f(x, y, t)=2tx^3y^3-6t^2xy^3-6t^2x^3y$
with depicting in some different cases.
There are some interesting papers for handling (1+1) fractional diffusion equation by different methods, for example, by Sinc-Legendre collocation method, \cite{Saadatmandi2}, tau approach, \cite{Saadatmandi1}, differential transform method, \cite{Vedat}, and (2+1) fractional heat conduction equation by
Legendre polynomials method, \cite{Hamed}.
In the present paper,  we first compute the Lie symmetry algebra of (\ref{heateq}) equation and then
 one-dimensional optimal system of subalgebras is obtained. A reduction of FDE (\ref{heateq}) by its similarity solution is calculated.
In the last section we present a numerical result via the spectral method based on Bernstein's operational matrices to obtain a solution for the 2D fractional heat equation with some initial conditions.
\section{Symmetry analysis}
\begin{thm}
The Lie symmetry algebra of the fractional two-dimensional heat conduction equation (\ref{heateq}) is spanned by following vector fields:
\begin{eqnarray}\nonumber
{\bf X}_1={\partial\over\partial t},~~~ {\bf X}_2={\partial\over\partial x},~~~ {\bf X}_3={\partial\over\partial y},~~~ {\bf X}_4=u{\partial\over\partial u},\\\label{eq:vectorfield2}
{\bf X}_5=\alpha t{\partial\over\partial t}+\beta x{\partial\over\partial x}+\beta y{\partial\over\partial y},~~~{\bf X}_6=g(t, x, y){\partial\over\partial u}.
\end{eqnarray}
\end{thm}
\begin{proof}
To find the Lie symmetries of fractional 2D heat equation (\ref{heateq}),
assume that (\ref{heateq}) is invariant under the following one parameter group transformation
\begin{equation}\label{eq:onepara}
\begin{aligned}
\overline{t}&=t+\epsilon T(t, x, y, u)+{\mathcal O}(\epsilon^2),\\
\overline{x}&=x+\epsilon X(t, x, y, u)+{\mathcal O}(\epsilon^2),\\
\overline{y}&=y+\epsilon Y(t, x, y, u)+{\mathcal O}(\epsilon^2),\\
\overline{u}&=u+\epsilon U(t, x, y, u)+{\mathcal O}(\epsilon^2),\\
D^{\alpha}_{\overline{t}}\overline{u}&=D^{\alpha}_t+\epsilon U^t_{\alpha}+{\mathcal O}(\epsilon^2),\\
D^{\beta}_{\overline{x}}\overline{u}&=D^{\beta}_xu+\epsilon U^x_{\beta}+{\mathcal O}(\epsilon^2),\\
D^{\beta}_{\overline{y}}\overline{u}&=D^{\beta}_yu+\epsilon U^y_{\beta}+{\mathcal O}(\epsilon^2),
\end{aligned}
\end{equation}
where $T, X, Y$ and $U$ are infinitesimals and $U^t_{\alpha}, U^x_{\beta}$ and $U_y^{\beta}$ are prolonged infinitesimals of the order $\alpha$ and $\beta$ respectively. Suppose
that
\begin{eqnarray}\label{eq:vectorfield1}
{\bf X}=T(t, x, y ,u){\partial\over\partial t}+X(t, x, y ,u){\partial\over\partial x}+Y(t, x, y ,u){\partial\over\partial y}+U(t, x, y ,u){\partial\over\partial u}
\end{eqnarray}
is the infinitesimal generator corresponding to  equation  (\ref{heateq}) under the one parameter group transformation  (\ref{eq:onepara}).
The $(\alpha, \beta)$-order prolongation of vector field $X$, \cite{Leo}, is defined by
\begin{eqnarray}\label{prolongation}
{\bf X}^{(\alpha, \beta)}={\bf X}+U^{t}_{\alpha} \partial_{D^{\alpha}_{t}}+ U^{x}_{\beta} \partial_{D^{\beta}_{x}}+ U^{y}_{\beta} \partial_{D^{\beta}_{y}} + \cdots
\end{eqnarray}
where
\begin{eqnarray}\label{u0}
\begin{aligned}
U^{t}_{\alpha}&=D^{\alpha}_{t}(U-Xu_x-Yu_y-Tu_t)+ X D^{\alpha}_{t}u_x+Y D^{\alpha}_{t}u_y+T D^{\alpha+1}_{t}u,\\
&= D^{\alpha}_{t}U-D^{\alpha}_{t}(Xu_x)-D^{\alpha}_{t}(Yu_y) -D^{\alpha}_{t}(Tu_t)+ X D^{\alpha}_{t}u_x+Y D^{\alpha}_{t}u_y+T D^{\alpha+1}_{t}u \\
U^{x}_{\beta}&=D^{\beta}_{x}(U-Xu_x-Yu_y-Tu_t)+X D^{\beta}_{x}u_{x}+Y D^{\beta}_{x}u_y+ TD^{\beta}_{x}u_t,\\
&= D^{\beta}_{x}U -D^{\beta}_{x}(Xu_x) - D^{\beta}_{x}(Yu_y) - D^{\beta}_{x}(Tu_t)+X D^{\beta}_{x}u_{x}+Y D^{\beta}_{x}u_y+ TD^{\beta}_{x}u_t\\
U^{y}_{\beta}&=D^{\beta}_{y}(U-Xu_x-Yu_y-Tu_t)+X D^{\beta}_{y}u_{x}+Y D^{\beta}_{y}u_y+ TD^{\beta}_{y}u_t,\\
&= D^{\beta}_{y}U -D^{\beta}_{y}(Xu_x) -D^{\beta}_{y}(Yu_y) -D^{\beta}_{y}(Tu_t)+X D^{\beta}_{y}u_{x}+Y D^{\beta}_{y}u_y+ TD^{\beta}_{y}u_t.
\end{aligned}
\end{eqnarray}
Notice that the symbol $D_{t}$ is the total derivative operator with respect to $t$: $$D_{t}=\frac{\partial}{\partial{t}}+u_{t}\frac{\partial}{\partial{u}}+u_{xt}\frac{\partial}{\partial{u_{x}}}+u_{tt}\frac{\partial}{\partial{u_{t}}}+\cdots,$$
and the operator $D^{\alpha}_{t}, D^{\beta}_{x}$ and $D^{\beta}_{y}$ are the total fractional derivative operators. By the Leibnitz rules [12]
\begin{eqnarray}
\begin{aligned}\label{lebnitz}
D^{\alpha}_{t}(f(t)g(t))&=&\sum^{\infty}_{n=0}\binom{\alpha}{n} D^{\alpha-n}_{t}f(t)g^{(n)}(t),\;\;\; \alpha>0,\\
D^{\beta}_{x}(h(x)l(x))&=&\sum^{\infty}_{n=0}\binom{\beta}{n} D^{\beta-n}_{x}h(x)l^{(n)}(x),\;\;\; \beta>0,\\
D^{\beta}_{y}(r(y)z(y))&=&\sum^{\infty}_{n=0}\binom{\beta}{n} D^{\beta-n}_{y}r(y)z^{(n)}(y),\;\;\; \beta>0,
\end{aligned}
\end{eqnarray}
where
\begin{eqnarray*}
\binom{\alpha}{n}=\frac{\Gamma(1+\alpha)}{\Gamma(1+\alpha-n)\Gamma(n+1)},\;\;\;\;\; \binom{\beta}{n}=\frac{\Gamma(1+\beta)}{\Gamma(1+\beta-n)\Gamma(n+1)}.
\end{eqnarray*}
Applying (\ref{lebnitz}) into (\ref{u0}) we obtain
\begin{eqnarray}\label{u1}
\begin{aligned}
U^{t}_{\alpha}&= D^{\alpha}_{t}(U)-uD^{\alpha}_{t}(U_u)+U_uD^{\alpha}_{t}(u) +\sum_{n=1}^\infty\binom{\alpha}{n} D^{\alpha}_{t}(U_u)D^{\alpha-n}_{t}(u)
-\alpha(u_tT_u+T_t)D^{\alpha}_{t}(u)-\\
&\sum_{n=1}^\infty\binom{\alpha}{n}\left(({\alpha-n\over n+1}) D^{\alpha-n}_{t}(u)D^{n+1}_t(T) + D^{\alpha-n}_{t}(u_y)D^{n}_t(Y) + D^{\alpha-n}_{t}(u_x)D^n_t(X)\right)\\
U^{x}_{\beta}&= D^{\beta}_{x}(U)-uD^{\beta}_{x}(U_u)+U_uD^{\beta}_{x}(u) +\sum_{n=1}^\infty\binom{\beta}{n} D^{\beta}_{x}(U_u)D^{\beta-n}_{x}(u)
-\beta(u_xX_u+X_x)D^{\beta}_{x}(u)-\\
&\sum_{n=1}^\infty\binom{\beta}{n}\left(({\beta-n\over n+1}) D^{\beta-n}_{x}(u)D^{n+1}_x(X) + D^{\beta-n}_{x}(u_y)D^{n}_x(Y)+ D^{\beta-n}_{x}(u_t)D^n_x(T)\right)\\
U^{y}_{\beta}&= D^{\beta}_{y}(U)-uD^{\beta}_{y}(U_u)+U_uD^{\beta}_{y}(u) +\sum_{n=1}^\infty\binom{\beta}{n} D^{\beta}_{y}(U_u)D^{\beta-n}_{y}(u)
-\beta(u_yX_u+X_y)D^{\beta}_{y}(u)-\\
&\sum_{n=1}^\infty\binom{\beta}{n}\left(({\beta-n\over n+1}) D^{\beta-n}_{y}(u)D^{n+1}_y(Y) + D^{\beta-n}_{}(u_x)D^{n}_y(X)+ D^{\beta-n}_{y}(u_t)D^n_y(T)\right).\\
\end{aligned}
\end{eqnarray}
By symmetry criteria condition we have
\begin{eqnarray}\label{symcrit}
{\bf X}^{(\alpha, \beta)}\left[D^\alpha_t u-D^\beta_x u-D^\beta_y u-f\right]_{D^\alpha_t u=D^\beta_x u+D^\beta_y u+f}=0,
\end{eqnarray}
then we have
\begin{eqnarray}\label{symcrit1}
U^\alpha_t -U^\beta_x -U^\beta_y - T f_t - X f_x - Y f_y =0.
\end{eqnarray}
By putting (\ref{u1}) into (\ref{symcrit1}),  we get the following determining equations:
\begin{eqnarray}\label{sys}
\begin{aligned}
&T_x=T_y=T_u=0,\;\;\;\;\; X_t=X_y=X_u=0,\;\;\;\;\; Y_t=Y_x=Y_u=0,\;\;\;\;\; U_{ttu}=U_{xxu}=U_{yyu}=U_{uu}=0\\
& D^{\alpha}_{t}(U) +C_1uD^{\beta}_{x}(U_u) +C_2uD^{\beta}_{y}(U_u) -C_1D^{\beta}_{x}(U) -C_2D^{\beta}_{y}(U) -uD^{\alpha}_{t}(U_u)+\alpha fT_t+Xf_x+Yf_y+Tf_t=0 \\
& -C_2\sum_{n=3}^\infty\binom{\beta}{n} D^{n}_{y}(U_u)D^{\beta-n}_{y}(u) - C_1\sum_{n=3}^\infty\binom{\beta}{n} D^{n}_{x}(U_u)D^{\beta-n}_{x}(u)+\sum_{n=1}^\infty\binom{\alpha}{n} D^{n}_{t}(U_u)D^{\alpha-n}_{t}(u) \\
&+C_2\sum_{n=3}^\infty\binom{\beta}{n}\left(({\beta-n\over n+1}) D^{\beta-n}_{y}(u)D^{n+1}_y(Y) + D^{\beta-n}_{y}(u_y)D^{n}_y(Y)+ D^{\beta-n}_{y}(u_t)D^n_y(T)\right)\\
& +C_1\sum_{n=3}^\infty\binom{\beta}{n}\left(({\beta-n\over n+1}) D^{\beta-n}_{x}(u)D^{n+1}_x(X) + D^{\beta-n}_{x}(u_y)D^{n}_y(X)+ D^{\beta-n}_{x}(u_t)D^n_x(T)\right)\\
& +\sum_{n=3}^\infty\binom{\alpha}{n}\left(({\alpha-n\over n+1}) D^{\alpha-n}_{t}(u)D^{n+1}_t(T) + D^{\alpha-n}_{t}(u_y)D^{n}_t(Y) + D^{\alpha-n}_{t}(u_x)D^n_t(X)\right)=0
\end{aligned}
\end{eqnarray}
by solving the (\ref{sys}) we find
\begin{eqnarray}\label{eq:coeff}
T(t)=c_5\beta t+c_1,~~ X(x)=c_5\alpha x+c_2,~~Y(y)=c_5\alpha y+c_3,U(t, x, y ,u)=g(t, x, y)+c_3u.&
\end{eqnarray}
\end{proof}
Therefore the symmetry algebra of the fractional two-dimensional heat conduction equation (\ref{heateq}) is spanned by  vector fields (\ref{eq:vectorfield2}).
The commutation relations satisfied by generators (\ref{eq:vectorfield2})  are shown in Table \ref{tab1}.
\begin{table}
\caption{Commutation relations satisfied by infinitesimal generators (\ref{eq:vectorfield2})}\label{tab1}
\begin{tabular}{llllll} \hline
$[{\bf X}_i,{\bf X}_j]$ ~~~~~~&~~~~~~${\bf X}_1$ ~~~~~~~&~~~~~ $~{\bf X}_2$ ~~~~~~&~~~~~~${\bf X}_3$~~~~~~&~~~~~~${\bf X}_4$~~~~~~~&~~~~~ ${\bf X}_5$\\
\hline
${\bf X}_1$ ~~~~~~&~~~~~~ $0$ ~~~~~~&~~~~~~ $0$~~~~~~&~~~~~~ $0$ ~~~~~~&~~~~~~ $0$ ~~~~~~&~~~~~~ $\alpha{\bf X}_1$ \\[2mm]
${\bf X}_2$ ~~~~~~&~~~~~~ $0$ ~~~~~~&~~~~~~ $0$ ~~~~~~&~~~~~~ $0$ ~~~~~~&~~~~~~ $0$ ~~~~~~&~~~~~~ $\beta{\bf X}_2$ \\[2mm]
${\bf X}_3$ ~~~~~~&~~~~~~ $0$ ~~~~~~&~~~~~~ $0$ ~~~~~~&~~~~~~ $0$ ~~~~~~&~~~~~~ $0$ ~~~~~~&~~~~~~ $\beta{\bf X}_3$ \\[2mm]
${\bf X}_4$ ~~~~~~&~~~~~~ $0$ ~~~~~~&~~~~~~ $0$ ~~~~~~&~~~~~~ $0$ ~~~~~~&~~~~~~ $0$ ~~~~~~&~~~~~~ $0$ \\[2mm]
${\bf X}_5$ ~~~~~~&~~~~~~ $-\alpha{\bf X}_1$ ~~~~~~&~~~~~~ $-\beta{\bf X}_2$ ~~~~~~&~~~~~~ $-\beta{\bf X}_3$ ~~~~~~&~~~~~~ $0$ ~~~~~~&~~~~~~ $0$ \\
\hline
\end{tabular}
\end{table}
%
\section{The optimal system of one dimension subalgebras}
Let $G$ be a Lie group with corresponded Lie algebra $\mathfrak{g}$, i.e., $\mathfrak{g}=T_eG$.
There is an inner automorphism
$\tau_a\longrightarrow \tau \tau_a\tau^{-1}$ of the group $G$ for every arbitrary element $\tau\in G$. Every
an automorphism of the group $G$ induces an automorphism of $\mathfrak{g}$.
The set of all these automorphism forms a Lie group called {\it the adjoint group $G^A$}.
For arbitrary infinitesimal generators ${\bf X}$ and ${\bf Y}$ in $\mathfrak{g}$, the linear mapping ${\rm
Ad}~{\bf X}({\bf Y}):{\bf Y}\longrightarrow[{\bf X},{\bf Y}]$ is an automorphism of $\mathfrak{g}$,
called {\it the inner derivation of $\mathfrak{g}$}. The set of all these
inner derivations equipped by the Lie bracket $[{\rm Ad}{\bf X},{\rm Ad}{\bf Y}]={\rm Ad}[{\bf X},{\bf Y}]$ is a Lie
algebra $\mathfrak{g}^A$ called the {\it adjoint algebra of $\mathfrak{g}$}.
Two subalgebras in $\mathfrak{g}$ are {\it conjugate}
if there is a transformation of $G^A$ which takes one subalgebra
into the other. The collection of pairwise non-conjugate
$s$-dimensional subalgebras is called the {\it optimal system} of subalgebras
of order $s$ that was introduced by Ovsiannikov
\cite{Ovsiannikov}. Actually solving the optimal system problem is to determine the conjugacy inequivalent
subalgebras with the property that any other subalgebra is
equivalent to a unique member of the list under some element of
the adjoint representation i.e. $\overline{\mathfrak{h}}\,{\rm Ad(\tau)}\,\mathfrak{h}$ for some $\tau$ in a given Lie group.
The adjoint action is given by the Lie series
\begin{eqnarray}
{\rm Ad}(\exp(s\,{\bf X}_i)){\bf X}_j={\bf X}_j-s\,[{\bf X}_i,{\bf X}_j]+\frac{s^2}{2}\,[{\bf X}_i,[{\bf X}_i,{\bf X}_j]]-\cdots,
\end{eqnarray}
where $s$ is a parameter and $i,j=1,\cdots,n$. We can simplify a given arbitrary element,
\begin{eqnarray}\label{eq:vectorfield}
{\bf X}=\sum_{i=1}^5a_i{\bf X}_i,
\end{eqnarray}
of the Lie algebra $\mathfrak{g}=\langle {\bf X}_1, \cdots, {\bf X}_5 \rangle$. Note that the elements of $\mathfrak{g}$ can be represented by vectors
$(a_1, \ldots, a_5)\in{\Bbb R}^5$ since
each of them can be written in the form (\ref{eq:vectorfield}) for some constants $a_1, \ldots, a_5$.
Hence, the adjoint action can be regarded as (in
fact is) a group of linear transformations of the vectors $(a_1, \ldots, a_5)$.
\begin{thm}\label{thm:1}
Let $\mathfrak{g}=\langle {\bf X}_1, \cdots, {\bf X}_5 \rangle$ be the finite dimension Lie symmetry algebras of fractional heat equation (\ref{heateq}).
The optimal system of subalgebras
of order one is generated by
\begin{eqnarray*}
&&(1)\qquad \mathfrak{g}_1=\langle{\bf X}_1+a_2{\bf X}_2+a_3{\bf X}_3+a_4{\bf X}_4\rangle \\
&&(2) \qquad \mathfrak{g}_2=\langle {\bf X}_3+b_4{\bf X}_4\rangle\\
&&(3)\qquad \mathfrak{g}_3=\langle{\bf X}_2+c_4{\bf X}_4\rangle\\
&&(4) \qquad \mathfrak{g}_4=\langle d_4{\bf X}_4+d_5{\bf X}_5\rangle\\
&&(5) \qquad \mathfrak{g}_5=\langle{\bf X}_2+e_3{\bf X}_3+e_4{\bf X}_4+e_5{\bf X}_5\rangle, \quad e_3\neq0 \\
&&(6) \qquad \mathfrak{g}_6=\langle {\bf X}_1+f_3{\bf X}_3+f_4{\bf X}_4+f_5{\bf X}_5\rangle, \quad f_3\neq0\\
&&(7) \qquad \mathfrak{g}_7=\langle{\bf X}_1+g_2{\bf X}_2+g_4{\bf X}_4+g_5{\bf X}_5\rangle, \quad g_2\neq0 \\
&&(8) \qquad \mathfrak{g}_8=\langle {\bf X}_1+h_2{\bf X}_2+h_3{\bf X}_3+h_4{\bf X}_4+h_5{\bf X}_5\rangle, \quad h_2, h_3\neq0
\end{eqnarray*}
where $a_i$, $b_4$, $c_4$, $d_i$'s, $e_i$'s, $f_i$'s, $g_i$'s and $h_i$'s belong ${\Bbb R}$ are arbitrary constants.
\end{thm}
\begin{proof}
The function $F_i^s:\mathfrak{g}\to\mathfrak{g}$ defined by ${\bf X}\mapsto {\rm Ad}(\exp(s_i{\bf X}_i).{\bf X})$ is a linear map, for $i = 1,\cdots ,5$. The matrices $M_i^s$
of $F_i^s$ with respect to basis $\{{\bf X}_1,\cdots ,{\bf X}_5\}$ are
\begin{eqnarray*}
&&
M_1^{s_1}=\left[
\begin{array}{ccccccc}
1 & 0 &0 & 0 & 0\\
0 & 1 &0 & 0 & 0\\
0 & 0 &1 & 0 & 0\\
0 & 0 &0 & 1 & 0\\
-s_1 & 0 &0 & 0 & 1\\
\end{array}
\right],
\quad
M_2^{s_2}=\left[
\begin{array}{ccccccc}
1 & 0 &0 & 0 & 0\\
0 & 1 &0 & 0 & 0\\
0 & 0 &1 & 0 & 0\\
0 & 0 &0 & 1 & 0\\
0 & -s_2 &0 & 0 & 1\\
\end{array}
\right],
\quad
M_3^{s_3}=\left[
\begin{array}{ccccccc}
1 & 0 &0 & 0 & 0\\
0 & 1 &0 & 0 & 0\\
0 & 0 &1 & 0 & 0\\
0 & 0 &0 & 1 & 0\\
0 & 0 &-s_3 & 0 & 1\\
\end{array}
\right],
\\
&&M_4^{s_4}={\bf I}_{5}, \qquad
M_5^{s_5}=\left[
\begin{array}{ccccccc}
\exp(s_5) & 0 &0 & 0 & 0\\
0 & \exp(s_5) &0 & 0 & 0\\
0 & 0 &\exp(s_5) & 0 & 0\\
0 & 0 &0 & 1 & 0\\
0 & 0 &0 & 0 & 1\\
\end{array}
\right],
\end{eqnarray*}
where ${\bf I}_{5}$ is $5\times 5$ identity matrix.
Let ${\bf X}=\sum_{i=1}^5a_i{\bf X}_i$ then it is seen that
\begin{eqnarray*}
&&F_5^{s_5}\circ F_4^{s_4}\circ \cdots \circ F_1^{s_1}: {\bf X}\mapsto[a_1\exp(s_5)]{\bf X}_1 + [a_2\exp(s_5)]{\bf X}_2 + [a_3\exp(s_5)]{\bf X}_3 + a_4{\bf X}_4+\\
&&[a_5-a_1s_1\exp(s_5)-a_2s_2\exp(s_5)-a_3s_3\exp(s_5)]{\bf X}_5.
\end{eqnarray*}
If $a_1\neq0$ then we can omit the coefficient of ${\bf X}_5$ by setting
$\displaystyle{s_1={a_5\over a_1}}$ and $s_2=s_3=s_5=0$.
Scaling ${\bf X}$, we can assume that $a_1=1$. So, ${\bf X}$ reduces to the case (1).
If $a_1=a_2=0$ and $a_3\neq0$ then by putting $\displaystyle{s_3={a_5\over a_3}}$ and $s_5=0$, we can assume $a_3=1$
by scaling ${\bf X}$ and therefore we find the case (2).
If $a_1=a_3=0$ and $a_2\neq0$ then one can vanish the coefficients of ${\bf X}_1, {\bf X}_3$ and ${\bf X}_5$ by setting
$\displaystyle{s_2={a_5\over a_2}}$, $s_5=0$.
Scaling ${\bf X}$, we can assume that $a_2=1$. So, ${\bf X}$ is reduced to the case (3).
If $a_1=a_2=a_3=0$ then ${\bf X}$ is reduced to the case (4).
If $a_1=0$ and $a_2, a_3\neq0$ then one can vanish the coefficients of ${\bf X}_1$ by setting
$\displaystyle{s_2={a_5\over a_2}}, \displaystyle{s_3=-{a_5\over a_3}}$ and $s_5=0$.
Scaling ${\bf X}$, we can assume that $a_2=1$. So, ${\bf X}$ is reduced to the case (5).
If $a_2=0$ and $a_1, a_3\neq0$ then one can vanish the coefficients of ${\bf X}_2$ by setting
$\displaystyle{s_1={a_5\over a_1}}, \displaystyle{s_3=-{a_5\over a_3}}$ and $s_5=0$.
Scaling ${\bf X}$, we can assume that $a_1=1$. So, ${\bf X}$ is reduced to the case (6).
If $a_3=0$ and $a_1, a_2\neq0$ then one can vanish the coefficients of ${\bf X}_3$ by setting
$\displaystyle{s_1={a_5\over a_1}}, \displaystyle{s_2=-{a_5\over a_2}}$ and $s_5=0$.
Scaling ${\bf X}$, we can assume that $a_1=1$. Then ${\bf X}$ changes to the case (7).
If $a_1, a_2, a_3\neq0$ then one can vanish the coefficients of ${\bf X}_3$ by setting
$\displaystyle{s_1={a_5\over a_1}}, \displaystyle{s_2=-{a_5\over a_2}}$ and $s_3=s_5=0$.
Scaling ${\bf X}$, we can assume that $a_1=1$ and thus ${\bf X}$ maps to the case (8).
\end{proof}
%
\section{Symmetry reduction}
The infinitesimal generators ${\bf X}_1, \cdots, {\bf X}_4$ provide trivial invariant solutions and hence we
focus to deduce the corresponded characteristic equation of vector field ${\bf X}_5$ for getting the reduction equation.
The corresponded characteristic equation of
${\bf X}_{5}=\displaystyle{\alpha t\frac{\partial}{\partial{t}}+\beta t\frac{\partial}{\partial{t}}+\beta y\frac{\partial}{\partial{y}}}$ is
\begin{equation*}
\frac{dt}{\alpha t}=\frac{dx}{\beta x}=\frac{dy}{\beta y}.
\end{equation*}
Solving above equation leads to the following similarity transformation
\begin{equation}\label{inv5}
u=\omega(\xi_1, \xi_2), \;\;\;\; \xi_1=xt^{-\frac{\beta}{\alpha}},\;\; \xi_2=yt^{-\frac{\beta}{\alpha}}.
\end{equation}
\begin{thm}
The similarity solution (\ref{inv5}) transforms the fractional heat equation
(\ref{heateq}) to
\begin{eqnarray}\label{heateqreduc}
\left({\mathcal P}^{1,n-\alpha}_{\frac{\alpha}{\beta}, \frac{\alpha}{\beta}}\omega\right)(\xi_1, \xi_2)=
t^{\alpha}(x^{-\beta}+y^{-\beta})\left({\mathcal P}^{1,n-\beta}_{1, \infty}\omega\right)(\xi_1, \xi_2),
\end{eqnarray}
here ${\mathcal P}^{\tau, \alpha}_{\gamma_1, \gamma_2}$ denotes the extended left-hand sided
Erdelyi-Kober fractional derivative operator, \cite{Singla3, Leo}, with following definition
\begin{align}\label{p}
\left({\mathcal P}^{\tau, \alpha}_{\gamma_1, \gamma_2}\omega\right)(z_1, z_2)
=\prod_{j=0}^{n-1}\left(\tau+j-\frac{1}{\gamma_1}z_1\frac{\partial}{\partial z_1}-\frac{1}{\gamma_2}z_2\frac{\partial}{\partial z_2}\right)\times\left({\mathcal K}^{\tau,\alpha}_{\gamma_1, \gamma_2}\omega\right)(z_1, z_2),
\end{align}
and here
\begin{align}\label{k}
\left({\mathcal K}^{\tau,\alpha}_{\gamma_1, \gamma_2}\omega\right)(z_1, z_2),
=\left\{
\begin{array}{lll}
\displaystyle{\frac{1}{\Gamma(\alpha)}\int^{\infty}_{1}(\theta-1)^{\alpha-1}\theta^{-(\tau+\alpha)}\omega(z_1\theta^{\frac{1}{\gamma_1}}, z_2\theta^{\frac{1}{\gamma_2}})d\theta}
& , & \alpha>0, \\[4mm]
\omega(z_1, z_2) & , & \alpha=0,
\end{array}\right.
\end{align}
and
\begin{align*}
n=\left\{
\begin{array}{lll}
[\alpha]+1&,&\alpha\notin {\Bbb N},\\
\alpha&,&\alpha\in {\Bbb N}.
\end{array}\right.
\end{align*}
\end{thm}
\begin{proof}
Substituting transformation (\ref{inv5}) into (\ref{heateq}) leads to
\begin{equation}\label{eq:22}
D^\alpha_t \omega(\xi_1, \xi_2)=D^\beta_x \omega(\xi_1, \xi_2) + D^\beta_y \omega(\xi_1, \xi_2),
\end{equation}
For $n-1<\alpha<n$ and $n-1<\beta<n$, $n=1, 2, 3, \cdots$, then the similarity transformation (\ref{eq:22}) becomes
\begin{eqnarray}\label{eq:23-1}
D^\alpha_t \omega(\xi_1, \xi_2)=\frac{1}{\Gamma(n-\alpha)}\frac{\partial^n}{\partial{t}^n}\int^{t}_{0}(t-s)^{n-\alpha-1}\omega(xs^{-\frac{\beta}{\alpha}}, ys^{-\frac{\beta}{\alpha}})ds,\\\label{eq:23-2}
D^\beta_x \omega(\xi_1, \xi_2)=\frac{1}{\Gamma(n-\beta)}\frac{\partial^n}{\partial{x}^n}\int^{x}_{0}(x-s)^{n-\beta-1}\omega(st^{-\frac{\beta}{\alpha}}, yt^{-\frac{\beta}{\alpha}})ds,\\\label{eq:23-3}
D^\beta_y \omega(\xi_1, \xi_2)=\frac{1}{\Gamma(n-\beta)}\frac{\partial^n}{\partial{y}^n}\int^{y}_{0}(y-s)^{n-\beta-1}\omega(xt^{-\frac{\beta}{\alpha}}, st^{-\frac{\beta}{\alpha}})ds.
\end{eqnarray}
By change of a variable $\displaystyle{\theta=\frac{t}{s}}$ then $ds=-\displaystyle{\frac{t}{\theta^2}d\theta}$, the equation (\ref{eq:23-1}) can be written as
\begin{equation*}
\begin{aligned}
D^\alpha_t \omega(\xi_1, \xi_2)&=
\frac{1}{\Gamma(n-\alpha)}\frac{\partial^n}{\partial{t}^n}\int^{\infty}_{1}(t-\frac{t}{\theta})^{n-\alpha-1}\omega(xs^{-\frac{\beta}{\alpha}}, ys^{-\frac{\beta}{\alpha}})\frac{t}{\theta^{2}}d\theta\\
&=\frac{\partial^n}{\partial{t}^n}\left[t^{n-\alpha}\frac{1}{\Gamma(n-\alpha)}\int^{\infty}_{1}(\theta-1)^{n-\alpha-1}\theta^{-(n-\alpha+1)}\omega(\xi_1\theta^{\frac{\beta}{\alpha}}, \xi_2\theta^{\frac{\beta}{\alpha}})d\theta\right]\\
&=\frac{\partial^n}{\partial{t}^n}\left[t^{n-\alpha}\left({\mathcal K}^{1,n-\alpha}_{\frac{\alpha}{\beta}, \frac{\alpha}{\beta}}\omega\right)(\xi_1, \xi_2)\right].
\end{aligned}
\end{equation*}
By the chain rule of differentiation yields
\begin{equation*}
\begin{aligned}
\frac{\partial^n}{\partial{t}^n}\left[(t^{n-\alpha}\left(K^{1,n-\alpha}_{\frac{\alpha}{\beta}, \frac{\alpha}{\beta}}\omega\right)(\xi_1, \xi_2)\right]&=
\frac{\partial^{n-1}}{\partial{t}^{n-1}}\left[\frac{\partial}{\partial{t}}\left(t^{n-\alpha}\left({\mathcal K}^{1,n-\alpha}_{\frac{\alpha}{\beta}, \frac{\alpha}{\beta}}\omega\right)(\xi_1, \xi_2)\right)\right]\\
&=\frac{\partial^{n-1}}{\partial{t}^{n-1}}\left[t^{n-\alpha-1}(n-\alpha-\frac{\beta}{\alpha}\xi_1\frac{\partial}{\partial\xi_1}-\frac{\beta}{\alpha}\xi_2\frac{\partial}{\partial\xi_2})
\times\left({\mathcal K}^{1,n-\alpha}_{\frac{\alpha}{\beta}, \frac{\alpha}{\beta}}\omega\right)(\xi_1, \xi_2)\right]\\
&=t^{-\alpha}\prod_{j=0}^{n-1}\left(1-\alpha+j-\frac{\beta}{\alpha}\xi_1\frac{\partial}{\partial\xi_1}-\frac{\beta}{\alpha}\xi_2\frac{\partial}{\partial\xi_2}\right)\times\left({\mathcal K}^{1,n-\alpha}_{\frac{\alpha}{\beta}, \frac{\alpha}{\beta}}\omega\right)(\xi_1, \xi_2)\\
&=t^{-\alpha}\left({\mathcal P}^{1,n-\alpha}_{\frac{\alpha}{\beta}, \frac{\alpha}{\beta}}\omega\right)(\xi_1, \xi_2),
\end{aligned}
\end{equation*}
here ${\mathcal P}^{1,n-\alpha}_{\frac{\alpha}{\beta}, \frac{\alpha}{\beta}}$ and ${\mathcal K}^{1,n-\alpha}_{\frac{\alpha}{\beta}, \frac{\alpha}{\beta}}$
denote the Erdelyi-Kober fractional derivative operators, \cite{Singla3, Leo}, by definitions (\ref{p}) and (\ref{k}).
Therefore
\begin{eqnarray}\label{pt}
D^\alpha_t \omega(\xi_1, \xi_2)=t^{-\alpha}\left({\mathcal P}^{1,n-\alpha}_{\frac{\alpha}{\beta}, \frac{\alpha}{\beta}}\omega\right)(\xi_1, \xi_2).
\end{eqnarray}
With similar calculations we obtain
\begin{eqnarray}\label{pxpy}
D^\beta_x \omega(\xi_1, \xi_2)=x^{-\beta}\left({\mathcal P}^{1,n-\beta}_{1, \infty}\omega\right)(\xi_1, \xi_2), \;\;\;\;\;\;\;\;
D^\beta_y \omega(\xi_1, \xi_2)=y^{-\beta}\left({\mathcal P}^{1,n-\beta}_{1, \infty}\omega\right)(\xi_1, \xi_2).
\end{eqnarray}
Putting (\ref{pt}) and (\ref{pxpy}) into (\ref{eq:22}),  our governed equation reduces  to (\ref{heateqreduc}).
\end{proof}
%
\section{Spectral method for (\ref{heateq}) equation}
\begin{figure}
\begin{center}
\includegraphics[width=7cm]{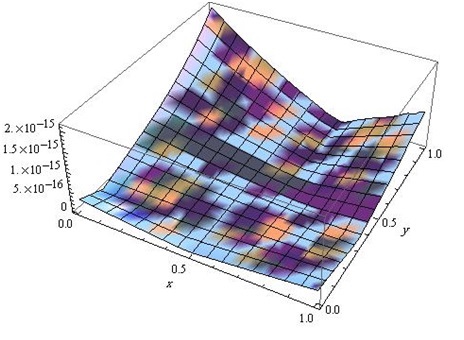}{(a)}
\includegraphics[width=7cm]{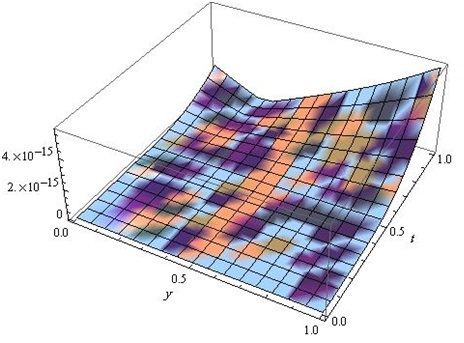}{(b)}\\
\includegraphics[width=7cm]{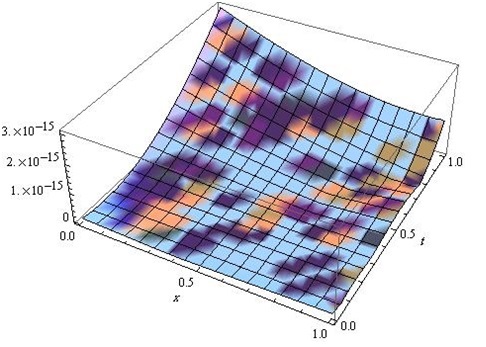}{(c)}
\end{center}
\caption{
(a) The absolute error at $t=0.5$, $x,y\in[0,1]$ for $\alpha= 1, \beta=2$ and $M=4$,
(b) The absolute error at $x=0.5$, $t,y\in[0,1]$ for $\alpha= 1, \beta=2$ and $M=4$,
(c) The absolute error at $y=0.5$, $t,x\in[0,1]$ for $\alpha= 1, \beta=2$ and $M=4$,
} \label{fig1}
\end{figure}
\begin{figure}
\begin{center}
\includegraphics[width=7cm]{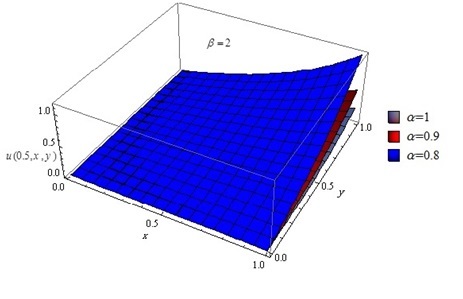}{(a)}
\includegraphics[width=7cm]{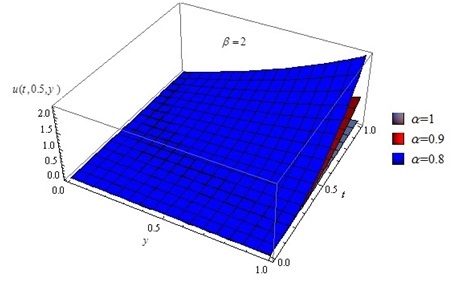}{(b)}\\
\includegraphics[width=7cm]{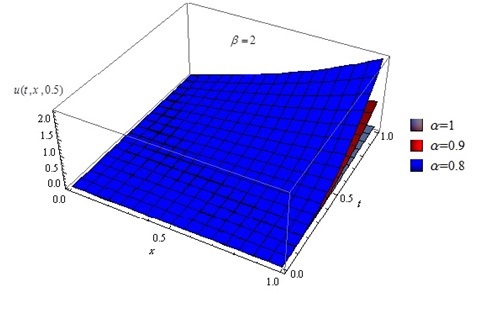}{(c)}
\end{center}
\caption{
(a) The approximate solutions at $t = 0.5$ for $\beta= 2$, $\alpha=0.8, 0.9, 1$ and $M=4$,
(b) The approximate solutions at $x = 0.5$ for $\beta= 2$, $\alpha=0.8, 0.9, 1$ and $M=4$,
(c) The approximate solutions at $y = 0.5$ for $\beta= 2$, $\alpha=0.8, 0.9, 1$ and $M=4$,
} \label{fig2}
\end{figure}
Now we apply the spectral method based on Bernstein operational matrices \cite{Alipour1, Alipour2, Alipour3, Alipour4, Alipour5} to solve the two-dimensional fractional heat equation
(\ref{heateq}) with following conditions:
\begin{eqnarray*}
&&u\left(0,x,y\right)=0,\ \ u\left(x,0,y\right)=0,\ \ u\left(t,x,0\right)=0,\\
&&\frac{\partial u(t,0,y)}{\partial x}=0,\ \ \frac{\partial u(t,x,0)}{\partial y}=0.
\end{eqnarray*}
In this method, $M$ is the number of members in Bernstein basis for any axes direction ($t, x$ and $y$).
First, we change the problem (\ref{heateq}) to the Caputo sense by the following relation.
\begin{eqnarray}
{}^C_0{D^{\alpha }_tf\left(t\right)}={}_0{D^{\alpha }_tf\left(t\right)-\sum^{n-1}_{k=0}{\frac{f^{\left(k\right)}\left(0\right)}{\Gamma \left(k-\alpha +1\right)}t^{k-\alpha },}}
\end{eqnarray}
so we have
\begin{eqnarray}
{}^C_0{D^{\alpha }_t}u\left(t,x,y\right)={}^C_0{D^{\beta }_x}u\left(t,x,y\right)+{}^C_0{D^{\beta }_y}u\left(t,x,y\right)+f\left(t,x,y\right).
\end{eqnarray}
Now, we apply the Bernstein operational matrices and propose the approximations as below:
\begin{eqnarray}
{}^C_0{D^{\alpha }_t}u\left(t,x,y\right)\approx {\psi }^T\left(t\right)K\widehat{\psi }\left(x,y\right).
\end{eqnarray}
Applying Riemann--Liouville fractional integration of order $\alpha $ with respect to $t$ on (\ref{heateq}) equation we have
\begin{eqnarray}
I^{\alpha }{}^C_0{D^{\alpha }_t}u\left(t,x,y\right)\approx {\psi }^T\left(t\right)P^{{\alpha }^T}K\widehat{\psi }\left(x,y\right),
\end{eqnarray}
where $P^{{\alpha }}$ is fractional integral operational matrix based on the one dimensional Bernstein basis ${\psi }\left(t\right)$. Then we get
\begin{eqnarray}
u\left(t,x,y\right)\approx {\psi }^T\left(t\right)P^{{\alpha }^T}K\widehat{\psi }\left(x,y\right)
\end{eqnarray}
Taking $\beta$ order derivative of $u(t, x, y)$ we get
\begin{eqnarray}
{}^C_0{D^{\beta }_x}u\left(t,x,y\right)\approx {\psi }^T\left(t\right)P^{{\alpha }^T}KH^{\left(\beta ,x\right)}\widehat{\psi }\left(x,y\right),
\end{eqnarray}
and
\begin{eqnarray}
{}^C_0{D^{\beta }_y}u\left(t,x,y\right)\approx {\psi }^T\left(t\right)P^{{\alpha }^T}KH^{\left(\beta ,y\right)}\widehat{\psi }\left(x,y\right),
\end{eqnarray}
where $H^{\left(\beta ,x\right)}$ and $H^{\left(\beta ,y\right)}$ are the Caputo fractional operational matrices with respect to $x$ and $y$, respectively based on the two dimensional Bernstein basis $\widehat{\psi }\left(x,y\right)$.
Also, we can obtain
\begin{eqnarray}\label{2}
f\left(t,x,y\right)\approx {\psi }^T\left(t\right)F\widehat{\psi }\left(x,y\right),
\end{eqnarray}
Using (\ref{2}) we get
\begin{eqnarray}
{\psi }^T\left(t\right)K\widehat{\psi }\left(x,y\right)={\psi }^T\left(t\right)P^{{\alpha }^T}KH^{\left(\beta ,x\right)}\widehat{\psi }\left(x,y\right)+{\psi }^T\left(t\right)P^{{\alpha }^T}KH^{\left(\beta ,y\right)}\widehat{\psi }\left(x,y\right)+{\psi }^T\left(t\right)F\widehat{\psi }\left(x,y\right),
\end{eqnarray}
which can be rewritten as
\begin{eqnarray}
{\psi }^T\left(t\right)\left(K-P^{{\alpha }^T}KH^{\left(\beta ,x\right)}-P^{{\alpha }^T}KH^{\left(\beta ,y\right)}-F\right)\widehat{\psi }\left(x,y\right)=0.
\end{eqnarray}
Hence it follows that
\begin{eqnarray}\label{3}
K-P^{{\alpha }^T}KH^{\left(\beta ,x\right)}-P^{{\alpha }^T}KH^{\left(\beta ,y\right)}-F=0,
\end{eqnarray}
Using the value of $K$ in (\ref{3}) we can get the approximate solution of the problem (\ref{heateq}) from (\ref{2}).
\begin{example}
Consider the following two-dimensional fractional heat conduction equation
\begin{eqnarray}
D^\alpha_t u(t, x, y)=D^\beta_x u(t, x, y)+D^\beta_y u(t, x, y)+2tx^3y^3-6t^2xy^3-6t^2x^3y,
\end{eqnarray}
where $0<\alpha \leq 1,\; 1<\beta \leq 2$ and $t, x,y\in [0,1]$. The exact solution for $\alpha= 1$ and $\beta= 2$ is
$$u(t, x, y)=t^2x^3y^3$$
In Figure \ref{fig1}, we calculate the absolute error at $t=0.5, x,y\in [0,1]$ and $x=0.5,t,y\in [0,1]$ and $y=0.5,t,x\in [0,1]$ and then the approximate solutions are in good agreement with exact solutions.
Also in Figure \ref{fig2}, we depict the approximate solutions for the fractional values of $\alpha$ and $\beta =2 $ at $t=0.5, x,y\in [0,1]$ and $x=0.5, t,y\in [0,1]$ and $y=0.5, t,x\in [0,1]$, respectively. Therefore, as $\alpha\to 1$ the approximate solutions equal to the exact solutions as expected.

It is notable that the used PC is Intel(R) Core(TM) i7-7700K CPU 4.20GHz. Also we apply version 13 of the Mathematica software for obtaining the results.
\end{example}
\section*{Data Availability}
Data sharing is not applicable to this article as no new data were
created or analyzed in this study.


\begin{thebibliography}{99}
%
\bibitem{Alipour1}
M. Alipour, D. Rostamy and D. Baleanu, {\it Solving multi-dimensional fractional optimal control problems with inequality constraint by Bernstein polynomials operational matrices}, Journal of Vibration and Control 19(16): 2523-2540 ~(2012).
%
\bibitem{Alipour2}
M. Alipour, {\it Numerical solution for fractional differential equations and optimal control problems}, Ph.D. Thesis, Imam Khomeini International University, 2013, (In Persian).
%
\bibitem{Alipour3}
D. Baleanu, M. Alipour and H. Jafari, {\it The Bernstein operational matrices for solving the fractional quadratic Riccati differential equations with the Riemann-Liouville derivative}, Abstract and Applied Analysis  Article ID 461970~(2013)
%
\bibitem{Bluman}
G.W. Bluman, A.F. Cheviakov, and S.C. Anco, Applications of Symmetry Methods to Partial Differential Equations, Springer, New York, 2010.
%
\bibitem{Bokhari}
A. H. Bokhari, A. Y. Al-Dweik, F. D. Zaman, A. H. Kara, and F. M. Mahomed, {\it Generalization of the double reduction theory}, Nonlinear Anal. Real World Appl. 11~ 3763-3769~ (2010).
%
\bibitem{Buckwar}
E. Buckwar, and  Y. Luchko, {\it Invariance of a partial differential equation of fractional order under the Lie group of scaling
transformations}, J. Math. Anal. Appl. 227, 81-97~ (1998).
%
\bibitem{Davydovych}
V. Davydovych, {\it Lie Symmetry of the Diffusive Lotka-Volterra System
with Time-Dependent Coefficients}, Symmetry 10, 41~(2018) .
%
\bibitem{Demiray}
S. T. Demiray, Y. Pandir, and H. Bulut, {\it The analysis of the exact solutions of the space fractional coupled KD equations},
AIP Conf. Proc. 1648, 370013 ~(2015).
%
\bibitem{Vedat}
V. S. Erturk, and Shaher Momani, {\it Solving systems of fractional differential equations using differential transform method}, J. Comput. Appl. Math. 215~ 142-151~(2008).
%
\bibitem{Gazizov1}
R. K. Gazizov, and A. A. Kasatkin, {\it Construction of exact solutions for fractional order differential equations by the invariant subspace method}, Comput. Math. Appl. 66 ~576-584 ~(2013) .
%
\bibitem{Gazizov2}
R. K. Gazizov, A. A. Kasatkin, and S. Yu. Lukashchuk, {\it Symmetry properties of fractional diffusion equations}, Phys. Scr. T 136~  014016~ (2009).
%
\bibitem{Gungor}
F. G{\"u}ng{\"o}r, C. {\"O}zemir, {\it Lie symmetries of a generalized Kuznetsov-Zabolotskaya-Khokhlov equation}, J. Math. Anal. Appl. 423~ 623-638~ (2015).
%
\bibitem{Hashemi2}
M. S. Hashemi, {\it Group analysis and exact solutions of the time-fractional Fokker--Planck equation}, Phys. A Stat. Mech. Appl. 417~ 141-149~ (2015).
%
\bibitem{Hashemi}
M. S. Hashemi, F. Bahrami, and R. Najafi, {\it Lie symmetry analysis of steady-state fractional
reaction-convection-diffusion equation}, Optik 138 240-249~(2017).
%
\bibitem{Oberlack1}
J. N. Hau, M. Oberlack, and  G. Chagelishvili, {\it On the optimal systems of subalgebras for the equations of hydrodynamic stability analysis of
smooth shear flows and their group-invariant solutions}, J. Math. Phys. 58, 043101~ (2017).
%
\bibitem{Hilfer}
R. Hilfer, Applications of Fractional Calculus in Physics, World Scientific, Singapore~2000.
%
\bibitem{Hu}
J. Hu, Y. J. Ye, S. F. Shen, and J. Zhang, {\it Lie symmetry analysis of the time-fractional KdV-type equation}, Appl. Math. Comput. 233  439-444~(2014).
%
\bibitem{Kasatkin}
A. A. Kasatkin, {\it Symmetry properties for systems of two ordinary fractional differential equations}, Uta Math. J.4 65-75~(2012) .
%
\bibitem{Hamed}
H. Khalil, and R. Ali Khan, {\it A new method based on Legendre polynomials for solutions
of the fractional two-dimensional heat conduction equation}, Comput. Math. Appl. 67 1938-1953~(2014) .
%
\bibitem{Kumar}
S. Kumar, K. Singh, and R. K. Gupta, {\it Painleve analysis, Lie symmetries and exact solutions for (2 + 1)-dimensional
variable coefficients Broer-Kaup equations}, Commun. Nonlinear Sci. Numer. Simul. 17, 1529-1541 ~(2012).
%
\bibitem{Leo}
R. A. Leo, G. Sicuro, and P. Tempesta, {\it A foundational approach to the Lie theory for fractional order partial differential
equations}, Fractional Calculus Appl. Anal. 20, 212-231~ (2017).
%
\bibitem{Lukashchuk}
S. Yu. Lukashchuk, and A.V. Makunin, {\it Group classification of  nonlinear time-fractional diffusion equation with a source term}, Appl. Math. Comput. 257  335-343~ (2015).
%
\bibitem{Mirzaee1}
F. Mirzaee, and N. Samadyar, {\it Application of orthonormal Bernstein polynomials to construct a efficient scheme for solving fractional stochastic integro-differential equation},
Optik - International Journal for Light and Electron Optics, 132 ~262-273~(2017).
%
\bibitem{Mirzaee2}
F. Mirzaee, and S.  Alipour, {\it Fractional-order orthogonal Bernstein polynomials for numerical solution of nonlinear fractional partial Volterra integro-differential equations},
Mathematical Methods in the Applied Sciences 42 (6), 1870-1893 ~(2019).
%
\bibitem{Nadj1}
M. Nadjafikhah, and R. Bakhshandeh-Chamazkoti, {\it Symmetry group classification for general Burgers equation}, Commun Nonlinear Sci Numer Simulat 15 2303-2310 ~(2010).
%
\bibitem{Nadj2}
M. Nadjafikhah, R. Bakhshandeh-Chamazkoti, and A. Mahdipour-Shirayeh, {\it A symmetry classification for a class of (2+1)-nonlinear wave equation}, Nonlinear Analysis 71  5164-5169~(2009).
%
\bibitem{Olver}
P.J. Olver, Applications of Lie Groups to Differential Equations, Springer-Verlag, Heidelberg~1986.
%
\bibitem{Ovsiannikov}
L.V. Ovsiannikov, Group Analysis of Differential Equations, Academic Press, New York~1982.
%
\bibitem{Ouhadan}
A. Ouhadan, and E. H. Elkinani, {\it Exact solutions of time fractional kolmogorov equation by using Lie symmetry analysis},
 Journal of Fractional Calculus and Applications, 5  97-104~(2014).
%
\bibitem{Rehman}
K. U. Rehman, N. U. Saba, M. Y. Malik, and A. A. Malik, {\it Encountering heat and mass transfer mechanisms simultaneously
in Powell-Erying fluid through Lie symmetry approach}, Case Studies in Thermal Engineering 10 ~ 541-549~(2017).
%
\bibitem{Alipour4}
D. Rostamy, M. Alipour, H. Jafari, and D. Baleanu, {\it Solving multi-term orders fractional differential equations by
 operational matrices of BPs with convergence analysis}, Romanian Reports in Physics 65(2): 334-349~(2013).
%
\bibitem{Alipour5}
D. Rostamy, H. Jafari, M. Alipour and C. M. Khalique, {\it Computational method based on Bernstein operational matrices for
multi-order fractional differential equations}, Filomat 28(3): 591-601~(2014).
%
\bibitem{Saadatmandi1}
A. Saadatmandi, and M. Dehghan, {\it A tau approach for solution of the space fractional diffusion equation}, Comput. Math. Appl. 62~1135-1142~(2011) .
%
\bibitem{Saadatmandi2}
A. Saadatmandi, M. Dehghan, and M. R. Azizi, {\it The Sinc-Legendre collocation method for a class of fractional convection-diffusion
equations with variable coefficients}, Commun. Nonlinear Sci. Numer. Simul. 17~ 4125-4136~(2012) .
%
\bibitem{Sahadevan}
R. Sahadevan, and T. Bakkyaraj, {\it Invariant analysis of time fractional generalized Burgers and Korteweg-de Vries equations}, J. Math. Anal. Appl. 393~  341-347~(2012).
%
\bibitem{SahaRay1}
S. Saha Ray,  {\it Invariant analysis and conservation laws for the time fractional (2+ 1)-dimensional
Zakharov–Kuznetsov modified equal width equation using Lie group analysis}, Computers \& Mathematics
with Applications, 76(9)~ 2110-2118~(2018).
%
\bibitem{SahaRay3}
S. Sahoo and S. Saha Ray, {\it Lie symmetries analysis and conservation laws for the fractional Calogero-
Degasperis-Ibragimov-Shabat equation}, International Journal of Geometric Methods in Modern
Physics, vol.15, no.7, 1850110~(2018).
%
\bibitem{Singla1}
K. Singla and R. K. Gupta, {\it On invariant analysis of some time fractional nonlinear systems of partial differential equations},
I, J. Math. Phys. 57, 101504~(2016).
%
\bibitem{Singla2}
K. Singla and R. K. Gupta, {\it On invariant analysis of space-time fractional nonlinear systems of partial differential equations.
II}, J. Math. Phys. 58, 051503~(2017).
%
\bibitem{Singla3}
K. Singla and R. K. Gupta, {\it Generalized Lie symmetry approach for fractional order systems of
differential equations. III}, J. Math. Phys. 58, 061501~ (2017).
%
\bibitem{Singla4}
K. Singla and R. K. Gupta, {\it Space-time fractional nonlinear partial differential equations: Symmetry analysis and
conservation laws}, Nonlinear Dynamics. 89, 321-331 ~(2017).
%
\bibitem{Srivastava}
V. K. Srivastava, S. Kumar, M. K. Awasthi, and B. Kumar Singh, {\it Two-dimensional time fractional-order biological population model and its analytical solution},
Egyptian Journal of Basic and Applied Sciences, 1~ 71-76~(2012) .
%
\bibitem{Torrisi}
M. Torrisi, and R. Tracina, {\it Second-order differential invariants of a family of diffusion equations}, J Phys A: Math Gen 38:7519-7526~(2005) .
%
\bibitem{Tarasov}
V. E. Tarasov, Fractional Dynamics: Applications of Fractional Calculus to Dynamics of Particles, Fields and Media, Springer-Verlag, Berlin Heidelberg~ 2010.
%
\bibitem{SahaRay2}
Vinita and S. Saha Ray, {\it  Lie symmetry reductions, power series solutions and conservation laws of
the coupled Gerdjikov-Ivanov equation using optimal system of Lie subalgebra}, Zeitschrift für
angewandte Mathematik und Physik, 72(4), 1-18~(2021).
%
\bibitem{Wang}
G. W. Wang, X. Q. Liu, and Y. Y. Zhang, {\it Lie symmetry analysis to the time fractional generalized fifth-order KdV equation}, Commun. Nonlinear Sci. Numer. Simul. 18~ 2321-2326~(2013) .
%
\end{thebibliography}
\end{document}